\documentclass[12p]{amsart}
\usepackage{latexsym}
\usepackage{amssymb}
\usepackage{amsmath}
\usepackage{amscd}
\usepackage{graphicx}
\usepackage{float}
\usepackage{mathrsfs} 
\usepackage{enumerate}

\newtheorem{theorem}{Theorem}[section]
\newtheorem{proposition}[theorem]{Proposition}
\newtheorem{lemma}[theorem]{Lemma}
\newtheorem{corollary}[theorem]{Corollary}
\theoremstyle{definition}
\newtheorem{definition}[theorem]{Definition}
\theoremstyle{remark}
\newtheorem{example}[theorem]{Example}
\newtheorem{remark}[theorem]{Remark}

\DeclareMathOperator{\Arg}{Arg}
\DeclareMathOperator{\Bin}{Bin}
\DeclareMathOperator{\ccc}{CC}
\DeclareMathOperator{\cord}{co-ord}
\DeclareMathOperator{\Conv}{Conv}
\DeclareMathOperator{\GL}{GL}
\DeclareMathOperator{\inter}{int}
\DeclareMathOperator{\Log}{Log}
\DeclareMathOperator{\ord}{ord}
\DeclareMathOperator{\supp}{supp}
\DeclareMathOperator{\Tri}{Tri}
\DeclareMathOperator{\ver}{vert}
\DeclareMathOperator{\Vol}{Vol}

\begin{document}

\title{%
On the order map for hypersurface coamoebas%
}                     

\author{%
  Jens Forsg{\aa}rd and
  Petter Johansson}

\dedicatory{In memory of Mikael Passare, who continues to inspire.}
  
\address{Department of Mathematics \\ Stockholm University \\
SE-106 91 Stockholm, Sweden.}
\email{jensf@math.su.se, petterj@math.su.se}

\begin{abstract}
Given a hypersurface coamoeba of a Laurent polynomial $f$, it is an open problem to describe the structure of the set
of connected components of its complement. In this paper we approach this problem by introducing the lopsided coamoeba. 
We show that the closed lopsided coamoeba comes naturally equipped with an order map, i.e.\ a map from the set of connected
components of its complement to a translated lattice inside the zonotope of a Gale dual of the point configuration $\supp(f)$.  
Under a natural assumption, this map is a bijection.
Finally we use this map to obtain new results concerning coamoebas of polynomials of small codimension.
\end{abstract}

\maketitle

\section{Introduction}
\label{sec:Intro}

The amoeba $\mathcal{A}(f)$ of a Laurent polynomial
\begin{equation}
\label{eqn:polynomial}
f(z) = \sum_{\alpha\in A} c_\alpha z^\alpha \in \mathbb{C}[z_1^{\pm 1}, \dots, z_n^{\pm 1}]
\end{equation}
is defined as the image of the zero locus $V(f)\subset(\mathbb{C}^*)^n$ under the componentwise logarithm mapping, 
i.e.\ $\mathcal{A}(f) = \Log(V(f))$ where
$\Log\colon (\mathbb{C}^*)^n\rightarrow \mathbb{R}^n$ is given by $z\mapsto (\log|z_1|, \dots, \log|z_n|)$.
An important step in the study of amoebas was taken in \cite{FPT} with the introduction of the so-called \emph{order map}. 
This is an injective map, here denoted by $\ord$, from the set of connected components of the complement of the amoeba $\mathcal{A}(f)$, to the set 
of integer points in the Newton polytope $\Delta_f = \Conv(A)$. If $E$ denotes a connected component of the amoeba complement
$\mathcal{A}(f)^c$, then the $j$th component of $\ord(E)$ is given by the integral
\[
 \ord(E)_j = \frac{1}{(2\pi i)^n} \int_{\Log^{-1}(x)} \frac{z_j\,f'_j(z)}{f(z)}\,\frac{d z_1 \cdots d z_n}{z_1\cdots z_n}, \qquad x\in E.
\]
Evaluating $\ord(E)$ in the univariate case amounts to counting zeros of $f$ by the argument principle, yielding an 
analogous interpretation of $\ord$ for multivariate polynomials. With this in mind, it is not hard to see that the 
vertex set $\ver(\Delta_f)$ is always contained in the image of $\ord$, and furthermore it was shown in \cite{R} that 
any subset of $\mathbb{Z}^n\cap \Delta_f$
that contain $\ver(\Delta_f)$ appears as the image of the order map for some polynomial with the given Newton polytope. 
Thus, even though the image of $\ord$ is non-trivial to determine, this map gives a good understanding of the structure 
of the set of connected components of the complement of the amoeba $\mathcal{A}(f)$. In particular, we have the sharp lower and 
upper bounds on the cardinality of this set given by $|\ver(\Delta_f)|$ and 
$|\mathbb{Z}^n\cap \Delta_f|$ respectively. See \cite{Mik} and \cite{PT} for an overview of amoeba theory.

The coamoeba $\mathcal{A}'(f)$ of $f$ is defined as the image of $V(f)$ under the componentwise argument mapping, i.e.\ 
$\mathcal{A}'(f) = \Arg(V(f))$ where $\Arg\colon (\mathbb{C}^*)^n \rightarrow \mathbf{T}^n$ is given by 
$\Arg(z) = (\arg(z_1), \dots, \arg(z_n))$. It is sometimes useful to consider the multivalued $\Arg$-mapping, 
which yields the coamoeba as a multiple periodic subset of 
$\mathbb{R}^n$. The starting point of this paper is the problem of describing the structure of the set of connected 
components of the complement of the closed coamoeba. 
The progress so far is restricted to that an upper bound 
on the cardinality of this set is given by the normalized volume $n!\Vol(\Delta_f)$, see \cite{N1}.
However, there is no known analogy of the order map for amoebas. 
Our approach to this problem
is to introduce the \emph{lopsided coamoeba}. As the name is choosen to emphasize the analogy with amoebas, let us briefly recall the notion of 
\emph{lopsided amoeba} as introduced in \cite{P}.

For a point $x\in \mathbb{R}^n$, consider the list of the moduli of the monomials of $f$ at $x$,
\[
 f\{x\} = \Big[ e^{\log|c_{\alpha_1}| + \langle\alpha_1, x\rangle}, \dots, e^{\log|c_{\alpha_N}| + \langle\alpha_N, x\rangle}\Big],
\]
where $N = |A|$. This list is said to be \emph{lopsided} if one component is greater than the sum of the others. 
If $f\{x\}$ is lopsided, then $x\notin \mathcal{A}(f)$.
 The lopsided amoeba $\mathcal{LA}(f)$ is defined as the set of points $x\in \mathbb{R}^n$ such that $f\{x\}$ is \emph{not} 
 lopsided. There is an inclusion $\mathcal{A}(f) \subset \mathcal{LA}(f)$, and in particular each connected component of 
 $\mathcal{LA}(f)^c$ is contained in a unique connected component of $\mathcal{A}(f)^c$. Let us consider the relation between the lopsided 
 amoeba and the order map. If the list $f\{x\}$ is dominated in the sense of lopsidedness by the monomial with 
 exponent $\alpha$, then it follows by Rouch\'e's theorem that $\ord(E) = \alpha$.
Hence, while $\ord$ can map connected components of the complement of $\mathcal{A}(f)$ to elements in the set 
$(\mathbb{Z}^n\cap \Delta_f)\setminus A$,  when restricted to the set of connected components of $\mathcal{LA}(f)^c$ it becomes an 
injective map into the point configuration $A$. In this sense, the structure of the set of connected components of the 
complement of the lopsided amoeba is better captured by $A$ than by its Newton polytope $\Delta_f$.

We always assume that a half-space $H\subset \mathbb{C}$ is open and contains the origin in its boundary, that is 
$H = H_\phi =  \{z\in \mathbb{C}\,\mid\,\Re(e^{i\phi}z)>0\}$ for some $\phi \in \mathbb{R}$. For each point $\theta\in \mathbf{T}^n$, 
consider the list
\[
 f\langle\theta\rangle = \Big[e^{i(\arg(c_{\alpha_1}) + \langle\alpha_1, \theta\rangle)}, \dots, e^{i(\arg(c_{\alpha_N}) + \langle\alpha_N, \theta\rangle)}\Big],
\]
which we by abuse of notation also view as a set $f\langle\theta\rangle\subset S^1\subset \mathbb{C}$. We say that the list 
$f\langle\theta\rangle$ is \emph{lopsided} 
if there exist a half-space $H\subset \mathbb{C}$ such that, as a set, $f\langle\theta\rangle\subset \overline{H}$ but 
$f\langle\theta\rangle\not\subset \partial H$.
\begin{definition}
\label{dfn:LopsidedComplement}
The lopsided coamoeba $\mathcal{LA}'(f)$ is the set of points $\theta\in\mathbf{T}^n$ such that $f\langle\theta\rangle$ is \emph{not} lopsided.
\end{definition}
When necessary we will consider $\mathcal{LA}'(f)$ as a subset of $\mathbb{R}^n$. 

The main result of this paper is that we provide an order map for lopsided coamoebas. That is, we provide a map from the set
of connected components of the complement of the closed lopsided coamoeba, to a translated lattice inside a certain zonotope,
related to a \emph{Gale dual} of $A$, see Theorem \ref{thm:MapToZonotope}. 

As noted above, the image of the order map of the (lopsided) amoeba depends in an intricate manner on the
coefficients of the polynomial $f$. The order map which we provide for the lopsided coamoeba will, under a natural assumption, be 
a bijection. That is, the dependency on the coefficients of $f$ lies only in the translation of the lattice, and this dependency
is explicitly given in Theorem \ref{thm:MapToZonotope}. As a consequence, we are able to use this map
to obtain new results concerning the geometry of coamoebas. In particular, we give an affirmative answer
to a special case of a conjecture by Passare, see Corollary \ref{cor:Passare}.

Let us give a brief outline of the paper. Section~\ref{sec:prel} contains fundamental results in coamoeba theory, most of 
which are previously known. In Section~\ref{sec:loop} we will turn to lopsided coamoebas, considering their fundamental
properties and their relation to ordinary coamoebas. 
In Section~\ref{sec:gale} we provide the order map for the lopsided coamoeba.
In the last section we consider coamoebas of polynomials of codimension one and two, using the results of the previous sections.

\subsection{Notation}
We will use $\ccc(S)$ to denote the set of connected components of the complement of a set $S$, in its natural ambient 
space. That is, $\ccc(\mathcal{A}(f))$ denotes the set of connected components of the complement of the amoeba, which always are 
subsets of $\mathbb{R}^n$, while $\ccc(\mathcal{A}'(f))$ denotes the set of connected components of the complement of the coamoeba viewed on 
the real $n$-torus $\mathbf{T}^n$.
The transpose of a matrix $M$ is denoted by $M^t$. By $g_M$ we denote the greatest common divisor of the maximal minors of 
$M$. We use $e_i$ for the $i$th vector of the standard basis in any vector space, and $\langle\cdot,\cdot\rangle$ for the standard 
scalar product. $I_m$ denotes the unit matrix of size $m\times m$. We use that convention that $X\subset Y$ includes the 
case $X = Y$.

\subsection{Acknowledgements}
Our greatest homage is paid to Mikael Passare, whose absence is still felt. To him we owe our knowledge and intuition 
concerning coamoebas. We would like to thank August Tsikh, whose comments greatly improved the manuscript. The first author 
is deeply grateful to Thorsten Theobald and Timo de Wolff for their hospitality in Frankfurt, and helpful suggestions on the 
manuscript. We would like to thank Ralf Fr\"oberg for his comments and suggestions, and Johannes Lundqvist for his interest 
and discussions. We would also like to thank the referee, whose suggestions and remarks led to substantial improvements.

\section{Preliminaries}
\label{sec:prel}

 As implicitly stated in the introduction, the coamoeba of a hypersurface is in general not closed. Let $\Gamma$ be a 
 (not necessarily proper) subface of $\Delta_f$. The truncated polynomial with respect to $\Gamma$ is defined as
\[
f_\Gamma(z) = \sum_{\alpha\in A\cap\Gamma} c_\alpha z^\alpha.
\]
 It was shown in \cite{J2} and \cite{NS1} that the closure of a coamoeba is the union of all the coamoebas of its truncated 
 polynomials, that is
\begin{equation}
\label{eqn:TruncatedUnion}
\overline{\mathcal{A}'}(f)= \bigcup_{\Gamma\subset \Delta_f}\mathcal{A}'(f_\Gamma).
\end{equation}
We will refer to $\mathcal{A}'(f_\Gamma)$ as the coamoeba of the face $\Gamma$. If the above union is taken only over the proper 
subfaces $\Gamma$ of $\Delta_f$ one obtains the \emph{phase limit set} ${\mathcal{P}^\infty}(f)$ (see \cite{NS1}), and similarly if the union 
is taken only over the edges of $\Delta_f$ one obtains the \emph{shell} $\mathcal{H}(f)$ of $\mathcal{A}'(f)$ (see \cite{J} and \cite{N1}). For the 
latter we note that the coamoeba of an edge $\Gamma\subset\Delta_f$ consists of a family of parallel hyperplanes, whose 
normal is in turn parallel to $\Gamma$. It is natural to focus on 
$\overline{\mathcal{A}'}(f)$ rather than $\mathcal{A}'(f)$, the main reason 
being that the components of the complement of $\overline{\mathcal{A}'}(f)$, when viewed in $\mathbb{R}^n$, are convex. 
To see this, we give 
the following argument due to Passare. If $\Theta\subset \mathbb{R}^n$ is a connected component of the complement of 
$\overline{\mathcal{A}'}(f)$, then the function  $g(w) = 1/f(e^{iw})$ is holomorphic on the tubular domain $\Theta+i\mathbb{R}^n$. As it 
cannot be extended to a holomorphic function on any larger 
tubular domain, the convexity follows from Bochner's tube theorem \cite{Boch}.

By abuse of notation one identifies the index set $A$ with the matrix
\begin{equation}
\label{eqn:A}
 A = \left(\begin{array}{cccc}
            1 & 1 & \cdots & 1\\
            \alpha_1 & \alpha_2 & \cdots & \alpha_N
           \end{array}\right).
\end{equation}
We will restrict the term \emph{integer affine transformation} of $A$ to refer to a matrix $T\in \GL_n(\mathbb{Q})$ such that 
\[
\left(\begin{array}{cc} 1 & 0 \\ 0 & T \end{array}\right) A \in \mathbb{Z}^{(n+1)\times N}.
\]
The transformation $T$ induces a function $\mathbb{C}^A\rightarrow \mathbb{C}^{TA}$ by the monomial change of variables
\[
 z_j\mapsto z^{T_j},
\]
where $T_j$ denotes the $j$th row of $T$. 
With the notation $e^{x+i\theta} = (e^{x_1+i\theta_1}, \dots, e^{x_n+i\theta_n})$ we find that
\[
T(f_j)\big(e^{(x+i\theta){T}^{-1}}\big) = \langle c_j, e^{(x+i\theta){T}^{-1}TA_j}\rangle= \langle c_j, e^{(x+i\theta) A_j}\rangle = f_j\big(e^{x+i\theta}\big).
\]
Thus, a point $\theta\in \mathbb{R}^n$ belongs to $\mathcal{A}'(f_j)$ if and only if $({T}^{-1})^t\theta$ belongs to $\mathcal{A}'(T(f_j))$. 
We conclude the following relation previously described in \cite{NS2}.

\begin{proposition}
 As subsets of\/ $\mathbb{R}^n$, we have that $\mathcal{A}'(T(f))$ is the image of $\mathcal{A}'(f)$ under the linear 
 transformation $({T}^{-1})^t$.
\end{proposition}

\begin{corollary}
\label{cor:CoAmoebaTransformation}
As subsets of\/ $\mathbf{T}^n$, the coamoeba $\mathcal{A}'(T(f))$ consists of $|\det(T)|$ linearly transformed copies of $\mathcal{A}'(f)$.
\end{corollary}

\begin{proof}
The transformation $(T^{-1})^t$ acts with a scaling factor $1/|\det(T)|$ on $\mathbb{R}^n$, now consider a fundamental domain.
\end{proof}

Any point configuration $A$ can be shrunk, by means of an integer affine transformation, to a point configuration whose 
maximal minors are relatively prime \cite{GKZ1}. 

The polynomial $f$, and the point configuration $A$, is called \emph{maximally sparse} if $A = \ver(\Delta_f)$. If in 
addition $\Delta_f$ is a simplex, then $V(f)$ is known as a \emph{simple hypersurface}, and we will say that $f$ a 
\emph{simple polynomial}.
Let us describe the coamoeba of a simple hypersurface. Consider first when $\Delta_f$ is the standard 2-simplex. After a 
dilation of the variables, which corresponds to a translation of the coamoeba, we can assume that 
$f(z_1, z_2) = 1 + z_1 + z_2$. If the coamoebas of the truncated polynomials of the edges of $\Delta_f$ are drawn, with 
orientations given by the outward normal vectors of $\Delta_f$, then $\mathcal{A}'(f)$ consist of the interiors of the oriented 
regions, 
together with all intersection points.
\begin{figure}[h]
\centering
\includegraphics[width=40mm]{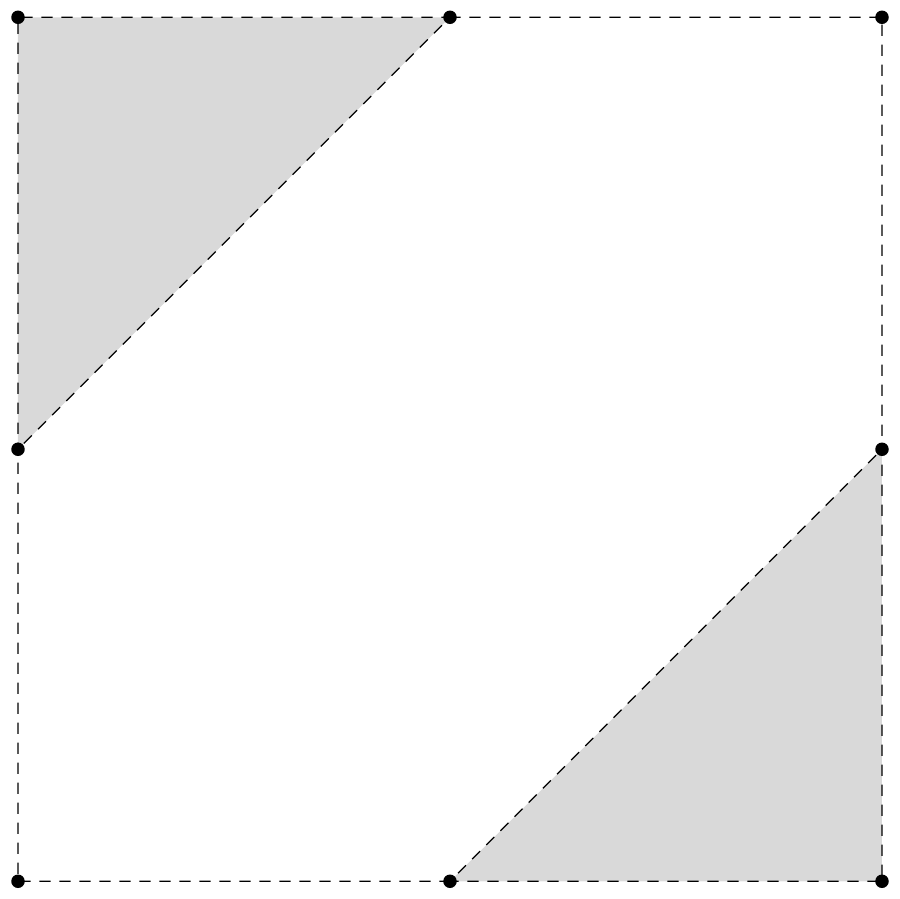}
\caption{The coamoeba of $f(z_1, z_2) = 1 +z_1+z_2$ in the domain $[-\pi,\pi]^2$}
\label{fig:Trinomial}
\end{figure}
An arbitrary simple trinomial differs from the standard 2-simplex only by an integer affine transformation, hence the 
coamoeba of any simple trinomial consists of a certain number of copies of $\mathcal{A}'(f)$, and is given by the same recipe as for 
the standard 2-simplex. 

Consider now when $\Delta_f$ is the standard $n$-simplex, that is $f(z) = 1 + z_1 + \dots + z_n$. Let $\Tri(f)$ denote the 
set of all trinomials one can construct from the set of monomials of $f$, which we still consider as polynomials in the 
$n$ variables $z_1, \dots, z_n$. It was shown in \cite{J} that we have the identity
\begin{equation}
\label{eqn:TrinomialUnion}
 \overline{\mathcal{A}'}(f) = \bigcup_{g\in \Tri(f)} \overline{\mathcal{A}'}(g),
\end{equation}
which also holds without taking closures if $n\neq 3$. Again, an arbitrary simple polynomial is only an integer affine 
transformation away, and hence the 
identity (\ref{eqn:TrinomialUnion}) holds for all simple hypersurfaces.

The complement of the closed coamoeba of $f(z) = 1 + z_1 + \dots + z_n$, in the fundamental domain $[-\pi,\pi)^n$ in $\mathbb{R}^n$, 
consists the convex hull of the open cubes $(0,\pi)^n$ and $(\pi, 0)^n$. In particular $\overline{\mathcal{A}'}(f)^c$ has exactly 
one connected component in $\mathbf{T}^n$.
Thus, the number of connected components of $\overline{\mathcal{A}'}(f)^c$ equals the normalized volume $n!\Vol(\Delta_f)=1$ in this 
case. For each integer affine transformation $T$ we have that $\Vol(\Delta_{T(f)}) = |\det(T)|\Vol(\Delta_f)$. It follows 
that for any simple hypersurface, the number of connected components of the complement of its coamoeba will be equal to the 
normalized volume of its Newton polytope.

Let us end this section with a fundamental property of the shell $\mathcal{H}(f)$, which we have not seen a proof of elsewhere.

\begin{lemma}
\label{lem:EdgeSubdivision}
Let $l\subset \mathbb{R}^n$ be a line segment with endpoints in $\overline{\mathcal{A}'}(f)^c$ that intersect 
$\overline{\mathcal{A}'}(f)$. Then $l$ 
intersect $\mathcal{A}'(f_{\Gamma})$ for some edge\/ $\Gamma\subset \Delta_f$. In particular, each cell of the hyperplane 
arrangement $\mathcal{H}(f)$ contains at most one connected component of $\overline{\mathcal{A}'}(f)^c$.
\end{lemma}

\begin{proof}
We have divided this rather technical proof into three parts.

\underline{Part 1:} Let us first present a slight modification of an argument given in \cite[Lemma 2.10]{J}, when proving the 
inclusion $\overline{\mathcal{A}'}(f) \subset \bigcup_{\Gamma\subset \Delta_f} \mathcal{A}'(f_\Gamma)$. Assume that 
$\Delta_f$ has full dimension and that the sequence
$\{z(j)\}_{j=1}^\infty\subset V(f)$ is such that
\[
\lim_{j\rightarrow \infty}z(j) \notin (\mathbb{C}^*)^n \quad \text{and} \quad  \lim_{j\rightarrow \infty} \Arg(z(j)) = \theta \in \mathbf{T}^n.
\]
We claim that $\theta\in \mathcal{A}'(f_\Gamma)$ for some strict subface $\Gamma \subset \Delta_f$. 
As $V(f)$ is invariant under multiplication of $f$ with a Laurent monomial, we can assume that the constant $1$ is a monomial of $f$.
We can also choose a subsequence of $\{z(j)\}_{j=1}^\infty$ such that, after possibly reordering $A$, 
\[
|z(j)^{\alpha_1}| \geq \dots \geq |z(j)^{\alpha_N}|, \quad j = 1, 2, \dots
\]
and in addition
\[
\lim_{j\rightarrow \infty} \frac{|z(j)|^{\alpha_k}}{|z(j)|^{\alpha_1}} \rightarrow d_k
\]
for some $d_k \in [0,1]$. It is shown in the proof of \cite[Lemma 2.10]{J} that $\Gamma = \{\alpha_k\,\mid\, d_k > 0\}$ is a face of 
$\Delta_f$, and furthermore that $\theta \in \mathcal{A}'(f_\Gamma)$. 
With the above ordering of $A$, assume that the constant $1$ is the $p$th monomial. We need to show that $\Gamma$ is a strict subface of $\Delta_f$. Assuming the contrary, we find that $d_k> 0$ for each $k$, and hence
\[
\lim_{j\rightarrow \infty} |z(j)^{\alpha_k}| = \lim_{j\rightarrow \infty} \frac{|z(j)|^{\alpha_k}}{|z(j)|^{\alpha_1}}\,|z(j)|^{\alpha_1} = \frac{d_k}{d_p},
\]
which in particular is finite and nonzero. As $\Delta_f$ has full dimension, this implies that that 
$\lim_{j \rightarrow \infty} |z(j)_m|$ is 
finite and non-zero for each $m = 1, \dots, n$. As $\arg(z(j)) \rightarrow \theta$ when $j\rightarrow \infty$, we find that 
$\lim_{j\rightarrow \infty} z(j) \in (\mathbb{C}^*)^n$, 
which contradicts our initial assumptions. Hence, $d_N = 0$, and $\Gamma$ is a strict subface of $\Delta_f$.

\underline{Part 2:} We now claim that if $n\geq 2$, then the set
\[
 P = \{z\in V(f) \,|\,\Arg(z) \in N(l)\cap\mathcal{A}'(f)\},
\]
where $N(l)$ is an arbitrarily small neighbourhood of $l$ in $\mathbb{R}^n$, is such that $\Log(P)$ is unbounded.
To see this, consider the function $g(w) = f(e^w)$, where $w_k = x_k + i \theta_k$. Notice that the $w$-space $\mathbb{C}^n$ 
is identified with the image of the $z$-space $(\mathbb{C}^*)^n$ under the multivalued, \emph{complex} logarithm. That is, 
the coamoeba $\mathcal{A}'(f)$ and the line $l$ are considered as subsets of $\mathbb{R}^n$, which is the image of the $w$-space $\mathbb{C}^n$
under taking coordinatewise imaginary parts.

We can assume
that $l$ is parallel to the $\theta_1$-axis and, by a translation of the coamoeba, that there are $\rho_1, \dots, \rho_n >0$  such that the set
\[
 S = [-\rho_1, \rho_1]\times \cdots \times [-\rho_n, \rho_n]
\]
fulfils $l\subset S \subset N(l)$. Furthermore we can choose $0<r<\rho_1$ such that, with
\[
 \tilde S = [-r, r] \times [-\rho_2, \rho_2]\times \cdots \times [-\rho_n, \rho_n],
\]
the set $S\setminus \tilde S$ consist of two $n$-cells that are neighbourhoods of the endpoints of $l$. Hence, we can assume that 
$S\setminus \tilde S\subset \overline{\mathcal{A}'}(f)^c$. If we assume that $\Log(P)$ is bounded, then there exists a sufficiently 
large $R\in \mathbb{R}$ such that if
\[
 D = \{x \in \mathbb{R}^n\,\mid\, |x|>R\},
\]
then $g(w)$ has no zeros in $D+i S \subset \mathbb{C}^n$. Let $w'$ denote the vector $(w_2, \dots, w_n)$, and let $(D+iS)'$ be the projection of 
$D+iS$ onto the last $n-1$ components.
Then in particular, $g(w)$ has no zeros when $w'\in (D+iS)'$ and $w_1$ lies in the domain given by 
$\{w_1 \,|\, r<|\Im(w_1)|<\rho_1\}\cup(\{w_1\,|\,|\Re(w_1)|>R\}\cap\{w_1\,|\,|\Im(w_1)|<\rho_1\})$, see Figure \ref{fig:IntegrationDomain}. Consider a 
curve $\gamma$ as in the figure, and the integral
\[
 k(w') = \frac{1}{2\pi i}\int_{\gamma} \frac{g'_1(w_1, w')}{g(w_1,w')} d w_1, \quad w'\in (D+iS)'
\]
By the argument principle, for a fix $w'$ this counts the number of roots of $g(w)$ inside the box in Figure \ref{fig:IntegrationDomain}. As it depend 
continuously on $w'$ in the domain $(D+iS)'$ it is constant, and by considering $w'$ with $|x'|>R$ (here it is essential 
that $n \geq 2$) we conclude that it is zero. However, this contradicts the assumption that $l$ intersects $\mathcal{A}'(f)$. 
Hence, $\Log(P)$ is unbounded.
\begin{figure}[h]
\centering
\includegraphics[width=60mm]{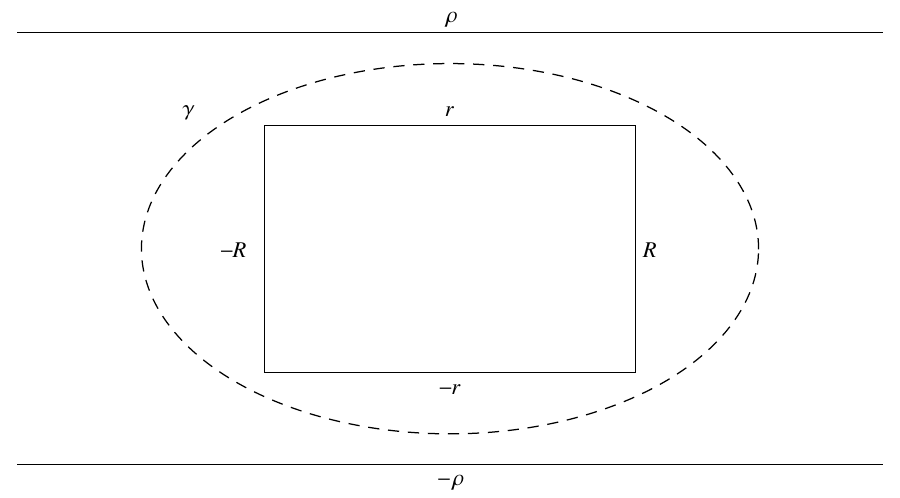}
\caption{The curve $\gamma \subset \mathbb{C}$.}
\label{fig:IntegrationDomain}
\end{figure}

\underline{Part 3:} We will now prove the lemma using induction on the dimension $d$ of $\Delta_f$. If $d=1$, then there is 
nothing to prove. Consider the case of a fix $d > 1$, assuming that the statement is proven for each smaller dimension. Notice that 
$f$ has $n-d$ homogeneities, and hence it is essentially a polynomial in $d$ variables. Dehomogenizing $f$  corresponds to 
projecting $\mathbf{T}^n$ onto $\mathbf{T}^d$ such that the coamoeba $\mathcal{A}'(f) \subset \mathbf{T}^n$ consist precisely of the fibers over the 
coamoeba of the dehomogenized polynomial. The image of $l$ under this projection will intersect the coamoeba of an edge of 
$\Delta_f$ in $\mathbf{T}^d$ if and only if $l$ intersect the coamoeba of an edge of $\Delta_f$ in $\mathbf{T}^n$. Hence, it is enough to 
prove the statement under the assumption that $d=n$. In particular, $n \geq 2$.

Choose a decreasing sequence $\{\varepsilon(k)\}_{k=1}^\infty$ of positive real numbers, such that 
$\lim_{k\rightarrow \infty} \varepsilon(k) = 0$,
and consider the family of neighbourhoods of $l$ given by
\[
N(l,k) = \Big\{\theta \in \mathbb{R}^n \,\Big\vert\, \inf_{x\in l} |\theta- x| < \varepsilon(k)\Big\},
\]
where $|\cdot|$ denotes the Euclidean norm on $\mathbb{R}^n$. Define
\[
P(k) =  \{z\in V(f) \,|\,\Arg(z) \in N(l,k)\cap\mathcal{A}'(f)\}.
\]
As $n \geq 2$, Part 2 shows that for each $k$, the set $\Log(P(k))$ is unbounded. That is, for each $k$, we can find a
sequence $\{z(k,m)\}_{m=1}^\infty$ such that $z(k,m) \in V(f)$, with 
\[
\Arg(z(k,m)) \in N(l,k)\cap \mathcal{A}'(f) \subset \overline{N(l,k) \cap \mathcal{A}'(f)},
\] 
however $\lim_{m\rightarrow \infty} z(k,m) \notin (\mathbb{C}^*)^n$. Since $\overline{N(l,k) \cap \mathcal{A}'(f)}$ is compact, we can 
choose a subsequence such
that $\Arg(z(k,m))$ converges to some $\theta(k) \in  \overline{N(l,k) \cap \mathcal{A}'(f)}$ when $m\rightarrow \infty$.
Then, Part 1 gives a strict subface $\Gamma(k)$ of $\Delta_f$ such that $\theta(k) \in \mathcal{A}'(f_{\Gamma(k)})$.
Since $\Delta_f$ has only finitely many strict subfaces, we can choose a subsequence of $\{\theta(k)\}_{k=1}^\infty$ such that $\Gamma = \Gamma(k)$
does not depend on $k$. As $\{\theta(k)\}_{k=1}^\infty \subset \overline{N(l,1)}$, which is compact, we can also choose
this subsequence such that $\theta(k)$ converges to some $\theta \in \overline{N(l,1)}$ when $k\rightarrow \infty$.
On the one hand, we have that $\theta \in l$ by construction of the sets $N(l,k)$. On the other hand, that 
$\theta(k) \in \mathcal{A}'(f_{\Gamma})$ implies that $\theta \in \overline{\mathcal{A}'}(f_\Gamma)$.
In particular, $l$ and $\overline{\mathcal{A}'}(f_\Gamma)$ intersect at $\theta$. 

The identity 
\eqref{eqn:TruncatedUnion} shows that the endpoints of $l$ is contained in the complement of $\overline{\mathcal{A}'}(f_\Gamma)$. As 
the dimension of $\Gamma$ is strictly less than the dimension of $\Delta_f$, the induction hypothesis shows that $l$ 
intersect the coamoeba of an edge of $\Gamma$. As each edge of $\Gamma$ is an edge of $\Delta_f$, the lemma is proven.
\end{proof}

\section{Lopsided coamoebas}
\label{sec:loop}

In this section we will investigate the basic properties of (closed) lopsided coamoebas. The formulation of Definition 
\ref{dfn:LopsidedComplement}
 was partly chosen to stress the analogy with the lopsided amoeba. A more natural description is perhaps the following; 
 denote the components of $f\langle\theta\rangle$ by $t_1, \dots, t_N$, and consider the convex cone
\[
\mathbb{R}_+f\langle\theta\rangle = \big\{r_1t_1 +\dots + r_Nt_N \,\,|\,\, r_1, \dots, r_N\in \mathbb{R}_+\big\}.
\]

\begin{lemma}
\label{lem:LopsidedCone}
 We have that $\theta\in \mathcal{LA}'(f)$ if and only if\/ $0\in \mathbb{R}_+f\langle\theta\rangle$.
\end{lemma}

\begin{proof}
 If $\theta\in \mathcal{LA}'(f)^c$, then $\mathbb{R}_+f\langle\theta\rangle\subset \inter(H)$, where $H\subset \mathbb{C}$ is the half-space such that 
 $f\langle\theta\rangle\subset H$ but $f\langle\theta\rangle\not\subset \partial H$. Conversely, if 
 $\mathbb{R}_+f\langle\theta\rangle$ does not contain the origin, then it follows from the convexity of $\mathbb{R}_+f\langle\theta\rangle$ that there exist a 
 half-space $H$ such that 
 $\mathbb{R}_+f\langle\theta\rangle\subset \inter(H)$.
\end{proof}

\begin{corollary}
\label{cor:Inclusion}
 We have the inclusion $\mathcal{A}'(f)\subset \mathcal{LA}'(f)$.
\end{corollary}

\begin{proof}
If $f(re^{i\theta}) = 0$ then $0\in\mathbb{R}_+f\langle\theta\rangle$.
\end{proof}

\begin{corollary}
 If $A$ is simple, then $\mathcal{A}'(f) = \mathcal{LA}'(f)$.
\end{corollary}

\begin{proof}
By considering integer affine transformations, we see that it is enough to prove this for the standard $n$-simplex 
$f(z) = 1 + z_1 + \dots +z_n$. 
We have that $0\in \mathbb{R}_+ f\langle\theta\rangle$ if and only if we can find $r_0, \dots, r_n\in \mathbb{R}_+$ such that 
$r_0 + r_1e^{i\theta_1} + \dots +r_n e^{i\theta_n} = 0$, and this is equivalent to $\theta \in \mathcal{A}'(f)$.
\end{proof}

Simple hypersurfaces are not the only ones for which the identity $\mathcal{A}'(f) = \mathcal{LA}'(f)$ holds. It will be the case as soon as 
$\mathcal{A}'(f) = \mathbf{T}^n$, and such examples are easy to construct by considering products of polynomials. An example of a nonsimple 
polynomial where $\mathcal{A}'(f) = \mathcal{LA}'(f)$, however $\mathcal{A}'(f)\neq \mathbf{T}^n$, is given by $f(z_1,z_2) = 1 + z_1 + z_2 - r z_1 z_2$ for any 
$r\in \mathbb{R}_+$.

Consider the polynomial
\[
 F(c,z) = \sum_{\alpha\in A} c_\alpha z^\alpha,
\]
obtained by viewing the coefficients $c$ as variables. This polynomial has a coamoeba 
$\mathcal{A}'(F)\subset \mathbf{T}^{N+n}$ which, as $F$ is simple, coincides
with its lopsided coamoeba $\mathcal{LA}'(F)$. As the convex cone $\mathbb{R}_+f\langle\theta\rangle$ coincides with the cone 
$\mathbb{R}_+ F\langle\arg(c), \theta\rangle$, we see that $\mathcal{LA}'(f)$ is
nothing but the intersection of $\mathcal{A}'(F)$ with the sub $n$-torus of $\mathbf{T}^{N+n}$ given by fixing $\Arg(c)$. 
In this manner, the lopsided coamoeba inherits 
some properties of simple coamoebas.

\begin{proposition}
\label{pro:IntrinsicIsTrinomialUnion}
Let\/ $\Tri(f)$ denote the set of all trinomials $g$ one can construct from the set of monomials of $f$. Then
\[
 \overline{\mathcal{LA}'}(f) = \bigcup_{g\in \Tri(f)} \overline{\mathcal{A}'}(g).
\]
\end{proposition}

\begin{proof}
By the previous discussion we can view $\mathcal{LA}'(f)$ is the intersection of ${\mathcal{A}'}_F$ with the sub $n$-torus of $\mathbf{T}^{N+n}$ given by 
fixing $\Arg(c)$. This
is of course also the case for each trinomial $g\in \Tri(f)$, and hence the identity follows from (\ref{eqn:TrinomialUnion}).
\end{proof}

As was the case in (\ref{eqn:TrinomialUnion}), this identity holds also without taking closures if $N\neq 4$. Lopsided 
coamoebas first appeared under this disguise in \cite{J}. This proposition gives a naive algorithm for determining lopsided 
coamoebas, by determining the coamoebas of each trinomial in $\Tri(f)$.

\begin{figure}[h]
\centering
\includegraphics[width=40mm]{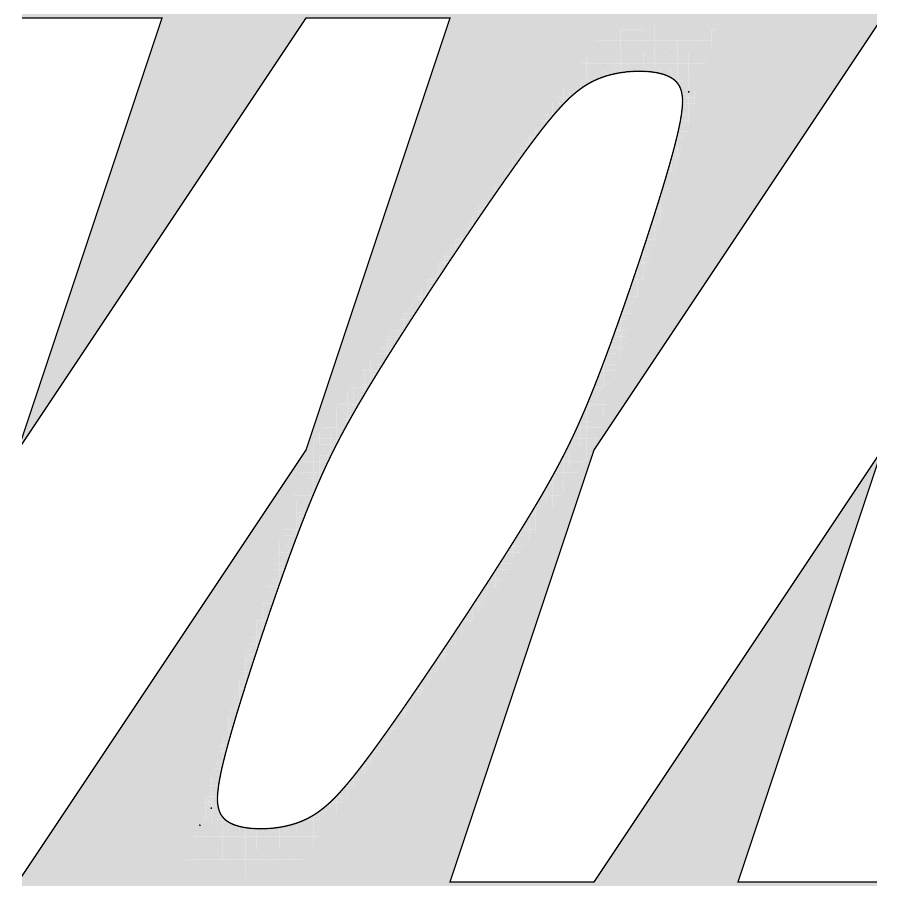}
$\qquad$
\includegraphics[width=40mm]{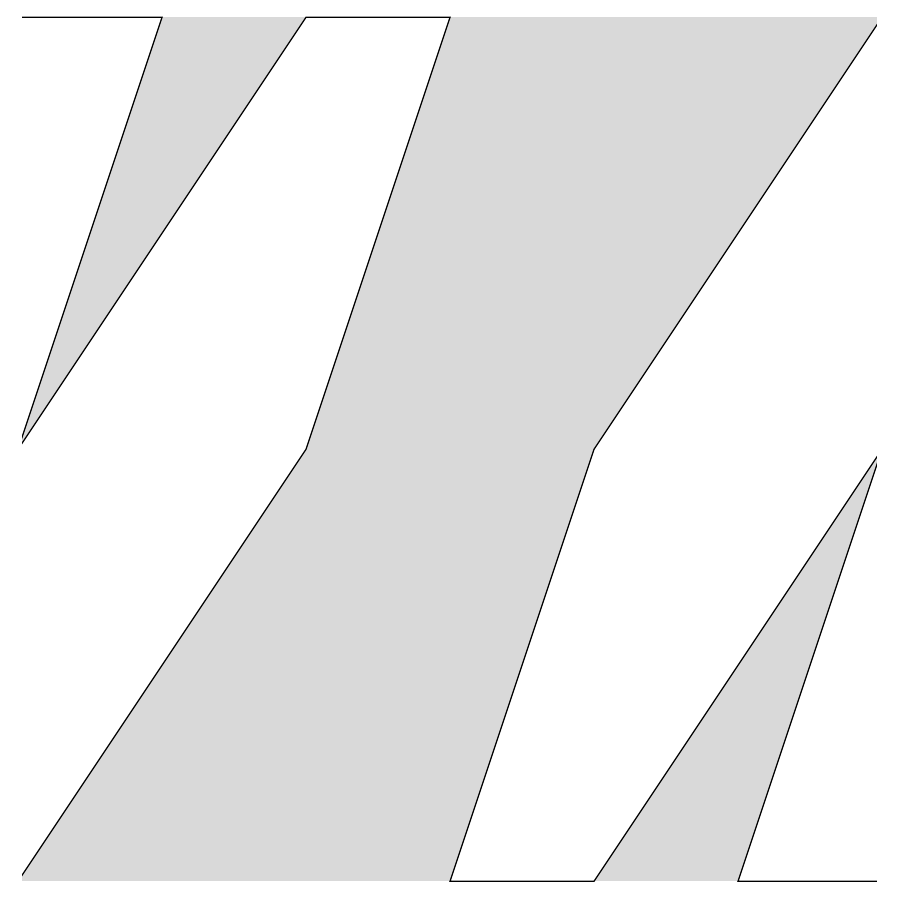}\\
\vspace{5mm}
\includegraphics[width=40mm]{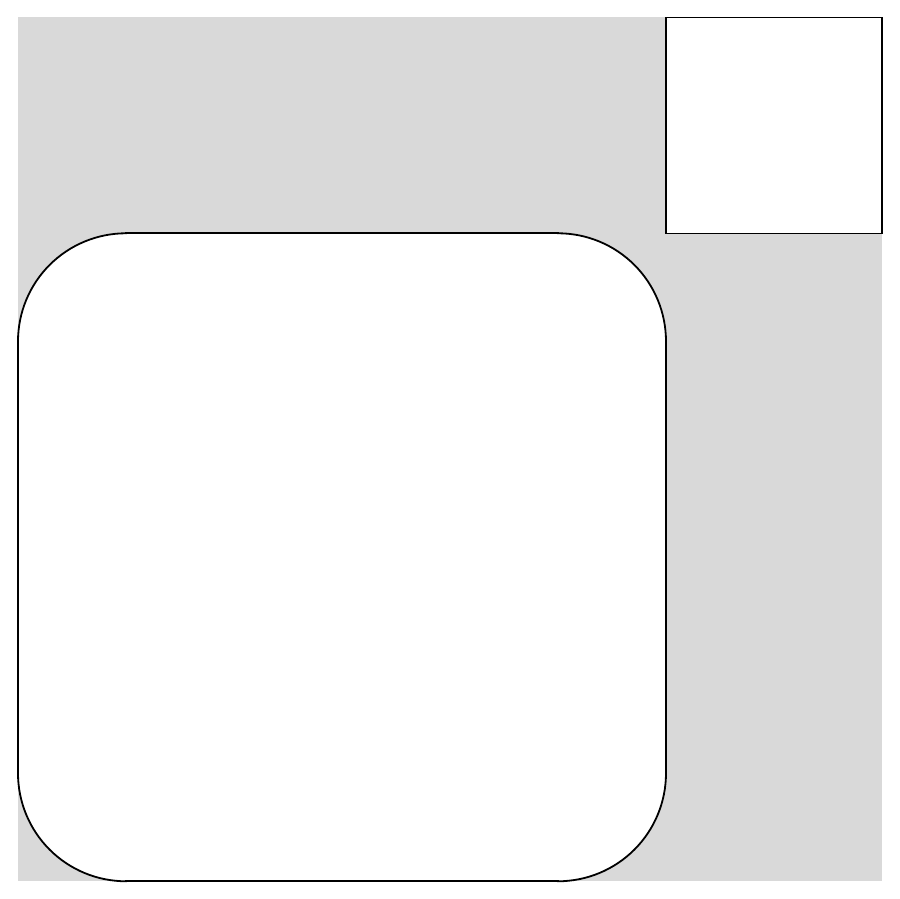}
$\qquad$
\includegraphics[width=40mm]{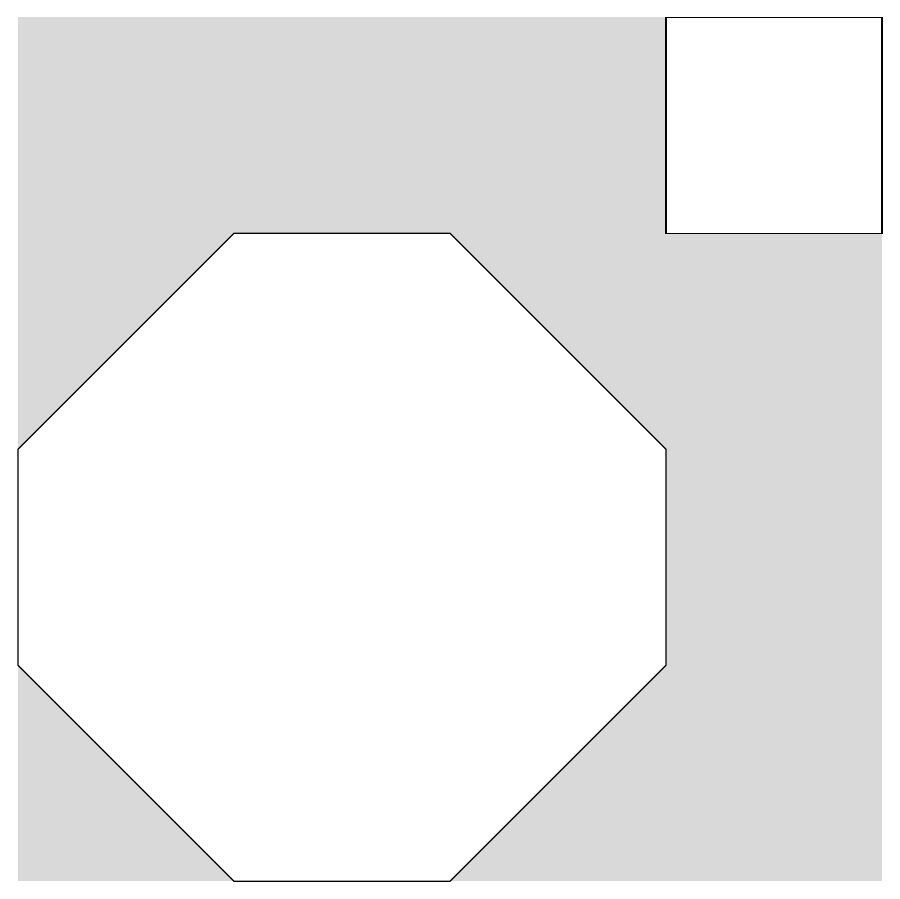}\\
\caption{Above: the coamoeba and lopsided coamoeba of $f(z_1, z_2) = z_1^3 + z_2 + z_2^2- z_1 z_2$. Below: the coamoeba and 
lopsided coamoeba of $f(z_1, z_2) = 1 + z_1 + z_2 + i z_1 z_2$.}
\label{fig:LopsidedExample}
\end{figure}
\begin{definition}
Let $\Bin(f)$ denote the set of all binomials that can be obtained by removing all but two monomials of $f$. The \emph{shell} $ \mathcal{LH}(f)$ of the lopsided coamoeba $\mathcal{LA}'(f)$ is defined as the union
\[
 \mathcal{LH}(f) = \bigcup_{g\in \Bin(f)} \mathcal{A}'(g)
\]
\end{definition}
In the case $n\geq 2$, Proposition \ref{pro:IntrinsicIsTrinomialUnion} states that $\overline{\mathcal{LA}'}(f)$ is the closure of the 
coamoeba of the polynomial $\prod_{g\in \Tri(f)}g(z)$. Recall that the ordinary shell of a coamoeba is defined as the union of all coamoebas
of the edges of its Newton polytope.
As the Newton polytope of each binomial in $\Bin(f)$ is an edge of 
the Newton polytope of some trinomial in $\Tri(f)$, we find that $ \mathcal{LH}(f)$ is a subset of the ordinary shell of this product, 
which motivates the choice of name.

\begin{proposition}
\label{pro:LopsidedBoundary}
The boundary of $\overline{\mathcal{LA}'}(f)$ is contained in $ \mathcal{LH}(f)$.
\end{proposition}

\begin{proof}
The boundary of $\overline{\mathcal{LA}'}(f)$ consists of points $\theta$ for which $f\langle\theta\rangle$ contains (at least) two antipodal 
points, which implies that $\theta$ belongs to the coamoeba of the corresponding binomial.
\end{proof}

The focus on $\overline{\mathcal{A}'}(f)$ rather than $\mathcal{A}'(f)$ leads us naturally to consider $\overline{\mathcal{LA}'}(f)$ in more detail. Its 
complement has the following
characterization.

\begin{proposition}
 We have that $\theta\in \overline{\mathcal{LA}'}(f)^c$ if and only if there is an open half-space $H\subset \mathbb{C}$ with 
 $f\langle\theta\rangle \subset H$.
\end{proposition}

\begin{proof}
The ``if'' part is clear. To show ``only if'', note that if $\theta\in \mathcal{LA}'(f)^c$, then there is an open half-space $H$ with 
$f\langle\theta\rangle \subset \overline{H}$. If there is no open half-space $H$ with $f\langle\theta\rangle\subset H$, then $f\langle\theta\rangle$ 
contains two antipodal points. Then we can find a simple trinomial $g\in \Tri(f)$ such that 
$\theta\in \overline{\mathcal{A}'}(g)$, and by the description of simple trinomials in the previous section there is a sequence 
$\{\theta_k\}_{k=1}^\infty\subset \inter(\mathcal{A}'(g))$ such that $\lim_{k\rightarrow \infty}\theta_k = \theta$. 
As $g$ is simple we have that $\mathcal{A}'(g)= \mathcal{LA}'(g)$, hence for each $\theta_k$ the list $g\langle\theta_k\rangle$ is not 
lopsided. Then neither is $f\langle\theta_k\rangle$, 
showing that $\{\theta_k\}_{k=1}^\infty\subset \mathcal{LA}'(f)$, and as a consequence that $\theta \in \overline{\mathcal{LA}'}(f)$.
\end{proof}

Let us end this section by describing the relation between the sets $\ccc(\overline{\mathcal{A}'}(f))$ and $\ccc(\overline{\mathcal{LA}'}(f))$, 
beginning with yet another characterization of $\mathcal{LA}'(f)$.

\begin{lemma}
\label{lem:VaryingCoefficients}
Let $f_r(z)$ denote the polynomial\/ $\sum_{\alpha\in A} r_\alpha c_\alpha z^\alpha$, where we have varied the moduli of the 
coefficients of $f$ by $r = (r_\alpha) \in \mathbb{R}_+^N$. Then
\[
 \mathcal{LA}'(f) = \bigcup_{r\in \mathbb{R}_+^N} \mathcal{A}'(f_r).
\]
\end{lemma}

\begin{proof}
The statement follows from Lemma \ref{lem:LopsidedCone}. If $\theta\in \mathcal{A}'(f_r)$, then $0\in \mathbb{R}_+ f_r\langle\theta\rangle$. 
Conversely, if $0\in \mathbb{R}_+ f\langle\theta\rangle$, then there exist an $r\in \mathbb{R}_+^N$ such that $f_r(e^{i\theta}) = 0$.
\end{proof}

\begin{proposition}
\label{pro:ComplementInjection}
 Each connected component of $\overline{\mathcal{A}'}(f)^c$ contains at most one connected component of $\overline{\mathcal{LA}'}(f)^c$. 
 \end{proposition}

\begin{proof}
It is clear that each connected component of $\overline{\mathcal{LA}'}(f)^c$ is included in some connected component of 
$\overline{\mathcal{A}'}(f)^c$, we only have to show that no two connected components of $\overline{\mathcal{LA}'}(f)^c$ are contained in the 
same connected component of $\overline{\mathcal{A}'}(f)^c$. We will show this by proving that any line segment $l$ with endpoints in 
$\overline{\mathcal{LA}'}(f)^c$ that intersect $\mathcal{LA}'(f)$, also intersect $\overline{\mathcal{A}'}(f)$.

Consider first the case when $f(z)$ is a univariate polynomial. Let $\theta_1, \theta_2 \in \overline{\mathcal{LA}'}(f)^c$ be the endpoints
of a line segment $l$, i.e.\ $l = [\theta_1, \theta_2]$,
and assume
that there exist a $\theta\in (\theta_1, \theta_2)$ with $\theta \in \mathcal{LA}'(f)$. Then Lemma \ref{lem:VaryingCoefficients} gives
an $r\in \mathbb{R}_+^N$ such that $\theta \in \mathcal{A}'(f_r)$. Let $\gamma$ be the path from $c$ to $rc$ in the coefficient space $(\mathbb{C}^*)^A$
given by
\[
 \gamma(t)_\alpha = r_\alpha^{1-t} c_\alpha, \quad t\in [0,1],
\]
and let $f_t$ denote the polynomial with coefficients $\gamma(t)$. Applying Lemma \ref{lem:VaryingCoefficients} once more, we find that 
for each $t \in [0,1]$ it holds that $\mathcal{A}'(f_t) \subset \mathcal{LA}'(f)$. In particular, for each $t$, we have that $\theta_1, \theta_2 \notin \mathcal{A}'(f_t)$.
Let $z\in \mathbb{C}^*$ denote a root of $f_0(z) = f_r(z)$ such that $\arg(z) = \theta$. It is well known that the roots of $f_t$ in $\mathbb{C}^*$ vary
continously with $r$. That is, we can find a continous path $t\mapsto z(t)$ in $\mathbb{C}^*$ such that $z(0) = z$ and
furthermore, for each $t\in [0,1]$ we have that $z(t)$ is a root of the polynomial $f_t(z)$. Notice that if $f_t(z)$ has a root of higher 
multiplicity at $z(t)$, then the path $t\mapsto z(t)$ is neither smooth nor
unique, however we need only that it is continous. Indeed, the continuity of the path $t\mapsto z(t)$ in $\mathbb{C}^*$ implies continuity of the path
$t \mapsto \arg(z(t))$. Finally, the continuity of the path $ t\mapsto \arg(z(t))$, together with the facts that $\theta_1, \theta_2 \notin \mathcal{A}'(f_t)$
for each $t\in [0,1]$ and that $\arg(z(0)) = \theta \in (\theta_1, \theta_2)$, implies that $\arg(z(t)) \in (\theta_1, \theta_2)$ for each $t$.
In particular, $\arg(z(1))\in (\theta_1, \theta_2)$, which proves the proposition in this case.

Consider now the case when $\Delta_f$ is one dimensional. Then $f(z)$ has $n-1$ quasi-homogeneities, and the coamoeba $\mathcal{A}'(f)$ consist
of a family of parallel hyperplanes, each orthagonal to $\Delta_f$. Dehomogenizing $f(z)$ corresponds to a projection $\mathbb{R}^n \rightarrow \mathbb{R}$
such that the hyperplanes in $\mathcal{A}'(f)$ are precisely the fibers over the points in the coamoeba of the dehomogenization $\tilde f$ of $f(z)$. 
This projection will map a line segment in $\mathbb{R}^n$ with endpoints in $\overline{\mathcal{LA}'}(f)^c$ that intersect $\mathcal{LA}'(f)$, to a line 
segment in $\mathbb{R}$ with endpoints
in $\overline{\mathcal{LA}'}(\tilde f)^c$ that intersect $\mathcal{LA}'(\tilde f)$. Hence, this case follows from the univariate case.

Now consider an arbitrary multivariate polynomial $f(z)$, and let $l$ be a line segment in $\mathbb{R}^n$ with endpoints
in $\overline{\mathcal{LA}'}(f)^c$ that intersect $\mathcal{LA}'(f)$.
By Lemma \ref{lem:VaryingCoefficients} there exists an $r\in \mathbb{R}_+^N$ such that $l$ intersect 
$\mathcal{A}'(f_r)$. Referring to Lemma \ref{lem:VaryingCoefficients} again, we find that $\overline{\mathcal{A}'}(f_r) \subset \overline{\mathcal{LA}'}(f)$, and hence
the endpoints of $l$ are contained in $\overline{\mathcal{A}'}(f_r)^c$. Applying Lemma \ref{lem:EdgeSubdivision} to the polynomial
$f_r$, we find an edge 
$\Gamma\subset \Delta_{f_r} = \Delta_f$ such that $l$ intersect $\mathcal{A}'((f_r)_\Gamma)$. This implies, by Lemma \ref{lem:VaryingCoefficients}, that
$l$ intersect $\mathcal{LA}'((f_r)_\Gamma) = \mathcal{LA}'(f_\Gamma)$. As the identity \eqref{eqn:TruncatedUnion} implies that 
$\overline{\mathcal{LA}'}(f_\Gamma) \subset \overline{\mathcal{LA}'}(f)$, we find that the endpoints of $l$ are contained in 
$\overline{\mathcal{LA}'}(f_\Gamma)^c$.
Since $\Gamma$ is one dimensional, we can conclude by the previous case that $l$ intersect $\mathcal{A}'(f_\Gamma)$. The identity \eqref{eqn:TruncatedUnion}
yields that $l$ intersect $\overline{\mathcal{A}'}(f)$.
\end{proof}

\section{The order map for the lopsided coamoeba}
\label{sec:gale}
The aim of this section is to provide an order map for the lopsided coamoeba.  The role played by the point configuration 
$A\subset \mathbb{Z}^n$ for the order map of the lopsided amoeba, is here given to a so-called \emph{dual matrix $B$}. Recall that 
$|A| = N$, that we are under the assumption that $\Delta_f$ is of full dimension, and that the integer $m = N - n - 1$ is 
the codimension of $A$. A dual matrix of $A$  is by definition an integer $N\times m$-matrix of full rank such that 
$AB = 0$. If in addition the columns of $B$ span the $\mathbb{Z}$-kernel of $A$, then $B$ is known as a \emph{Gale dual} of $A$. We 
denote by $\mathbb{Z}[B]\subset \mathbb{Z}^m$ the lattice generated by the \emph{rows} of $B$, and note that $B$ is a Gale dual of $A$ if 
and only if $\mathbb{Z}[B] = \mathbb{Z}^m$. In this manner, assuming that $B$ is a Gale dual will make our statements more streamlined, 
however it is not a necessary assumption in order to develop the theory.
We will label the rows of $B$ as $b_0, \dots, b_{n+m}$. The zonotope $\mathcal{Z}_B$ is defined as the set
\begin{equation}
\label{eqn:ZonotopeDescription}
 \mathcal{Z}_B = \bigg\{ \sum_{j=0}^{m+n} \frac\pi2 \,\mu_jb_j \,\bigg\vert \, |\mu_j|\leq 1, \,j=0, \dots, m+n\bigg\},
\end{equation}
see also \cite{Bf} and \cite{NP}.

Fix a polynomial $f \in (\mathbb{C}^*)^A$, i.e.\ with notation as in \eqref{eqn:polynomial} and \eqref{eqn:A} we fix a set of 
coefficients $c_{\alpha_1}, \dots, c_{\alpha_N}$.
Let us denote by  $\arg_\pi\colon \mathbb{C}^*\rightarrow (-\pi,\pi]$ the principal branch of the $\arg$-mapping, while $\Arg_\pi$ 
denotes the map acting on vectors componentwise by $\arg_\pi$.
\begin{lemma}
\label{lemma1}
For a fix polynomial $f$, and a fix point $\alpha\in A$ (that is, with the above notation $\alpha = \alpha_i$ for some $i$), 
consider the function $p_\alpha^k(\theta)$, with domain\/ $\mathbb{R}^n$, given by
\[
p^k_\alpha( \theta) = \arg_\pi\Big(\frac{c_{\alpha_k}e^{i\langle\alpha_k,\theta\rangle}}{c_\alpha e^{i\langle\alpha,\theta\rangle}}\Big)-\arg_\pi(c_{\alpha_k})+\arg_\pi(c_\alpha)-\langle\alpha_k-\alpha, \theta\rangle.
\]
Then $p^k_\alpha$ maps\/ $\mathbb{R}^n$ into\/ $2\pi \mathbb{Z}$, and furthermore it is locally constant off the coamoeba of the binomial 
$c_\alpha z^\alpha + c_{\alpha_k}z^{\alpha_k}$, as viewed in\/ $\mathbb{R}^n$.
\end{lemma}

\begin{proof}
For each $\theta$ we have that 
\[
\arg_\pi\Big(\frac{c_{\alpha_k}e^{i\langle\alpha_k,\theta\rangle}}{c_\alpha e^{i\langle\alpha,\theta\rangle}}\Big) = \arg_\pi(c_{\alpha_k}) - \arg_\pi(c_\alpha) + \langle\alpha_k - \alpha, \theta\rangle + 2 \pi j(\theta),
\]
where $j(\theta) \in \mathbb{Z}$. We see that $p_\alpha^k(\theta) = 2\pi j(\theta)$, and hence $p_\alpha^k$ maps $\mathbb{R}^n$ into 
$2\pi \mathbb{Z}$. It is clear that $j(\theta)$ is locally constant, as a function of $\theta$, off the set where 
$\arg_\pi\big(c_{\alpha_k}e^{i\langle\alpha_k,\theta\rangle}/c_\alpha e^{i\langle\alpha,\theta\rangle}\big) = \pi$. This set is 
precisely the coamoeba of the binomial $c_\alpha z^\alpha + c_{\alpha_k}z^{\alpha_k}$, as viewed in $\mathbb{R}^n$, which proves 
the second statement.
\end{proof}

In particular the vector valued function 
\[
p_\alpha(\theta) = \big(p_\alpha^1(\theta), \dots, p_\alpha^N(\theta)\big)
\]
is constant on each cell of the hyperplane arrangement $ \mathcal{LH}(f)$, considered as subsets of $\mathbb{R}^n$.

\begin{lemma}
\label{lemma2}
With notation as in the previous lemma, define $v_\alpha\colon \mathbb{R}^n \rightarrow \mathbb{R}^m$ by
\begin{equation}
\label{eqn:SecondLemma}
v_\alpha(\theta) = \big(\Arg_\pi(c)+p_\alpha(\theta)\big)B,
\end{equation}
where the multiplication with $B$ is usual matrix multiplication. Then $v_\alpha$ is well-defined on\/ $\mathbf{T}^n$ (i.e. it is periodic
in each $\theta_i$ with period $2\pi$), it is 
invariant under multiplication of $f$ by a Laurent monomial, and furthermore if $\theta \in \overline{\mathcal{LA}'}(f)^c$, 
then $v_\alpha(\theta) \in \inter(\mathcal{Z}_B)$.
\end{lemma}

\begin{proof}
For any $\theta \in \mathbb{R}^n$ we have that $\Arg_\pi(c)+p_\alpha(\theta)$ equals the vector 
\[\Big(\arg_\pi\Big(\frac{c_{\alpha_1}e^{i\langle\alpha_1,\theta\rangle}}{c_\alpha e^{i\langle\alpha,\theta\rangle}}\Big), \dots, \arg_\pi\Big(\frac{c_{\alpha_N}e^{i\langle\alpha_N,\theta\rangle}}{c_\alpha e^{i\langle\alpha,\theta\rangle}}\Big)\Big)
+ (\arg_\pi(c_\alpha)\langle\alpha, \theta\rangle, \theta_1, \dots, \theta_n)A,
\]
where $A$ denotes the matrix \eqref{eqn:A}. It follows that
\begin{equation}
\label{eqn:EvalutaionOfCoord}
\big(\Arg_\pi(c)+p_\alpha(\theta)\big)B = \Big(\arg_\pi\Big(\frac{c_{\alpha_1}e^{i\langle\alpha_1,\theta\rangle}}{c_\alpha e^{i\langle\alpha,\theta\rangle}}\Big), \dots, \arg_\pi\Big(\frac{c_{\alpha_N}e^{i\langle\alpha_N,\theta\rangle}}{c_\alpha e^{i\langle\alpha,\theta\rangle}}\Big)\Big)B.
\end{equation}
We conclude that $v_\alpha$ is well-defined on $\mathbf{T}^n$, and that it is 
invariant under multiplication of $f$ by a Laurent monomial.

Let us now turn to the last claim. Given a $\theta\in \overline{\mathcal{LA}'}(f)^c$, the components of $f\langle\theta\rangle$ are contained 
in one half-space $H\subset \mathbb{C}$. As $v_\alpha$ is invariant under multiplication of $f$ with a Laurent monomial, we can 
assume that $\alpha = 0$ and that $H=H_0$ is the right half space. That is
\[
 \arg_\pi(c_{\alpha_k}e^{i\langle\alpha_k, \theta\rangle}) = \frac \pi 2 \mu_k
\]
for some $\mu_k\in (-1, 1)$.
Since $\arg_\pi(x_1x_2) = \arg_\pi(x_1)+\arg_\pi(x_2)$ 
for any two elements $x_1, x_2\in H_0$, we find that
\[
 p_0^k(\theta) = \arg_\pi(c_{\alpha_k}e^{i\langle\alpha_k, \theta\rangle})-\arg_\pi(c_{\alpha_k}) - \langle\alpha_k, \theta\rangle.
\]
Thus, the following identities hold
\begin{equation}
\label{eq:MainProofSystem}
 \left\{\begin{array}{rcl}
 \arg_\pi(c_{\alpha_1})+\langle\alpha_1, \theta\rangle +  p_0^1(\theta)& = & \frac\pi2\mu_1\\
  & \vdots&\\
 \arg_\pi(c_{\alpha_N})+\langle\alpha_N, \theta\rangle + p_0^N(\theta)& = & \frac\pi2\mu_N.\\
\end{array}\right.
\end{equation}
Hence,
\[
 \big(\Arg_\pi(c)+p_\alpha(\theta)\big)B = \left(\frac\pi2\mu - (0, \theta_1, \dots, \theta_n)A\right)B = \frac\pi2\mu B   \in \inter(\mathcal{Z}_B).\qedhere
\]
\end{proof}

\begin{theorem}
\label{thm:MapToZonotope}
There is a well-defined map
\[
\cord\colon \ccc(\overline{\mathcal{LA}'}(f)) \rightarrow \inter(\mathcal{Z}_B) \cap (\Arg_\pi(c)B + 2\pi \mathbb{Z}[B]),
\]
which for\/ $\Theta \in  \ccc(\overline{\mathcal{LA}'}(f))$ is given by
\begin{equation}
\label{eqn:theorem43}
\cord(\Theta) = v_\alpha(\theta), \quad \theta\in \Theta, \alpha\in A.
\end{equation}
\end{theorem}

\begin{proof}
Note first that by Lemma \ref{lemma1} we have that $v_\alpha(\theta)\in \Arg_\pi(c) B + 2\pi \mathbb{Z}[B]$, and by Lemma \ref{lemma2} 
we have that $\theta \in \overline{\mathcal{LA}'}(f)^c$ implies that $v_\alpha(\theta) \in \inter(\mathcal{Z}_B)$. Hence we only need to show 
that the right hand side of \eqref{eqn:theorem43} is independent of $\theta \in \Theta$ and $\alpha \in A$, so that the 
given map is well-defined.

The first claim of Lemma \ref{lemma2} says that $v_\alpha$ is well-defined on $\mathbf{T}^n$. As the function $p_\alpha$ is 
constant on the cells of the hyperplane arrangement $\mathcal{H}_f$, Proposition \ref{pro:LopsidedBoundary} tells us that $v_\alpha$ is constant on the connected components 
of the complement of the lopsided coamoeba of $f$. That is, $v_\alpha(\theta)$ is independent of choice of 
$\theta \in \Theta$.

Finally, to see that $v_\alpha(\theta)$ is independent of the choice of $\alpha$, we note again that $v_\alpha(\theta)$ is 
invariant under multiplication of $f$ with a Laurent monomial. Hence we can assume that $f$ contains the monomial 
$\alpha = 0$, and that $H = H_0$. Then
\[
p_0^k(\theta) - p^k_\alpha(\theta) =\arg_\pi(c_\alpha e^{i\langle\alpha,\theta\rangle})-\arg_\pi(c_{\alpha})-\langle\alpha, \theta\rangle
\]
is independent of $k$, and hence $(p_0(\theta) - p_\alpha(\theta))B = 0$, which shows that $v_0(\theta) = v_\alpha(\theta)$ 
for each $\alpha$.
\end{proof}

\begin{definition}
The map $\cord$ from Theorem \ref{thm:MapToZonotope} is called the \emph{order map} of the lopsided coamoeba $\mathcal{LA}'(f)$.
\end{definition}

In order to show the statements on surjectivity and injectivity of $\cord$, we have to use a more detailed notation. After 
multiplication with a Laurent monomial, which neither affects the map $\cord$ nor the lopsided coamoeba $\mathcal{LA}'(f)$, we can 
assume that $A$ is of the form
\[
 A = \left(\begin{array}{ccc}
      1 & 1 & 1 \\
      0 & A_{\mathbf{1}} & A_{\mathbf{2}}
     \end{array}\right),
\]
where $A_{\mathbf{1}}$ is a non singular $n\times n$ matrix. We can also assume that $c_0= 1$, i.e. that the constant 1 is a monomial 
of $f$. 

The columns of any Gale dual of $A$ is a basis for its $\mathbb{Z}$-kernel. Hence, if we fix a Gale dual $\tilde B$, then any 
dual matrix can be presented in the form $B = \tilde B T$, for some $T\in \GL_m(\mathbb{Q})$. This implies that any dual matrix 
$B$ of $A$ can
be presented in the form
\begin{equation}
\label{eqn:ChoiceOfB}
B = \left(\begin{array}{c}
a_0\\
-A_{\mathbf{1}}^{-1}A_{\mathbf{2}}\\
I_m
\end{array}\right)T,
\end{equation}
where $a_0\in \mathbb{Q}^m$ is defined by the property that each column of $B$ should sum to zero, and $T\in \GL_m(\mathbb{Q})$.

\begin{lemma}
\label{lem:ThetaSystem}
Let $A$ be under the assumptions imposed above.
Let $c_{\mathbf{1}}$ and $c_{\mathbf{2}}$ denote the vectors $(c_1, \dots, c_{n})$ and $(c_{n+1}, \dots, c_{n+m})$ respectively, 
and similarly for $l\in \mathbb{Z}^N$ and 
$\mu\in \mathbb{R}^N$. Consider the system
\begin{equation}
\label{eq:SurjNInj1}
\left\{\begin{array}{lcl}
 \Arg_\pi(c_{\mathbf{1}})+\theta A_{\mathbf{1}} +  2\pi l_{\mathbf{1}}& = & \displaystyle\frac\pi2 \,\mu_{\mathbf{1}}  \vspace{5pt}\\
 \Arg_\pi(c_{\mathbf{2}})+\theta A_{\mathbf{2}} +  2\pi l_{\mathbf{2}}& = & \displaystyle\frac\pi2 \,\mu_{\mathbf{2}}
\end{array}\right.
\end{equation}
Then $\theta\in \overline{\mathcal{LA}'}(f)^c$ if and only if $\theta$ solves (\ref{eq:SurjNInj1}) for some integers $l$ and some 
numbers $\mu_0, \dots, \mu_{n+m}$ such that $\mu_0, \mu_1+\mu_0, \dots, \mu_{n+m}+\mu_0 \in (-1,1)$
\end{lemma}

\begin{proof}
If $\theta\in \overline{\mathcal{LA}'}(f)^c$, then there is a halfplane $H_\phi$ such that $f\langle\theta\rangle\subset H_\phi$. As the 
constant $1$ is a term of $f$, we can choose $\phi \in (-\pi/2, \pi/2)$. Considering the polynomial 
$e^{-i\phi}f(z)$, we find that this is lopsided at $\theta$ for $H_0$. Thus, there are numbers 
$\lambda_1, \dots, \lambda_{n+m} \in (-1, 1)$ and integers $l_1, \dots, l_{n+m}$ such that 
\[
\arg_\pi(c_k) + \langle\theta, \alpha_k\rangle + 2\pi l_k = \displaystyle\frac\pi 2 \,\lambda_k+\phi,\quad k=1, \dots, n+m.
\]
This shows that $\theta$ fulfils (\ref{eq:SurjNInj1}) with $l$ as above, $\mu_0 = -2\phi/\pi$ and $\mu_k = \lambda_k+2\phi/\pi$ for $k=1, \dots, n+m$. 
Conversely, if $\theta$ fulfils (\ref{eq:SurjNInj1}) for such $l$ and $\mu$, then $f\langle\theta\rangle\subset H_\phi$, where $\phi=-\pi\mu_0/2$.
\end{proof}

\begin{proposition}
\label{thm:Surjectivity}
The order map $\cord$ is a surjection. 
\end{proposition}

\begin{proof}
Let $A$ be under the assumptions imposed above.
Formally solving the first equation of (\ref{eq:SurjNInj1}) for $\theta$ by multiplication with $A_{\mathbf{1}}^{-1}$ and eliminating 
$\theta$ in the second equation, also applying the transformation $T$, one arrives at the equivalent system
 \begin{equation}
 \label{eq:SurjNInj2}
\left\{\begin{array}{rll}
 \theta & = &  \displaystyle\frac\pi 2\, \mu_{\mathbf{1}} A_{\mathbf{1}}^{-1} - \Arg(c_{\mathbf{1}})A_{\mathbf{1}}^{-1} - 2\pi l_{\mathbf{1}} A_{\mathbf{1}}^{-1}\vspace{5pt}\\
 \Arg_\pi(c)B + 2\pi (0, l_{\mathbf{1}}, l_{\mathbf{2}})B & = &  \displaystyle\frac \pi 2 \,(0,\mu_{\mathbf{1}}, \mu_{\mathbf{2}}) B.
\end{array}\right.
 \end{equation}
To see that $\cord$ is surjective, consider a point $\Arg_\pi(c) B + 2\pi lB  = \pi\lambda B/2 \in \inter(\mathcal{Z}_B)$, and 
note that we can assume that $l_0 = 0$. Define $\mu$ by $\mu_k = \lambda_k-\lambda_0$ for $k=0, \dots, n+m$. 
It follows that the pair $(l, \mu)$ fulfils the second equation of  (\ref{eq:SurjNInj2}). Let $\theta\in \mathbb{R}^n$ be defined 
by the first equation of ($\ref{eq:SurjNInj2}$), it then follows that the triple $(\theta, l,\mu$) fulfils 
($\ref{eq:SurjNInj1}$), and thus by Lemma \ref{lem:ThetaSystem} we have that $\theta\in \overline{\mathcal{LA}'}(f)^c$.
By tracing backwards we find that the order of the component of $\overline{\mathcal{LA}'}(f)^c$ containing $\theta$ is 
$\Arg_\pi(c) B + 2\pi lB$, and hence the map $\cord$ is surjective.
\end{proof}

\begin{proposition}
\label{thm:Injectivity}
If $g_A=1$, i.e.\ if the maximal minors of $A$ are relatively prime, then $\cord$ is an injection.
\end{proposition}

\begin{proof}
For any point $p\in \inter(\mathcal{Z}_B)$, the set of all $\mu \in \mathbb{R}^N$ such that $2\pi \mu B = p$, is an affine space, hence 
convex. It follows that the set of all $\mu\in (-1, 1)^N$ such that $2\pi\mu B=p$, being the intersection of two convex 
sets, is also convex. This implies that for fix integers $l$, the set of $\theta\in \mathbb{R}^n$ such that (\ref{eq:SurjNInj1}) is 
fulfilled with $\mu_0, \mu_1-\mu_0, \dots, \mu_N-\mu_0\in (-1,1)$ is in turn also convex, as it is the image of a convex set 
under an affine transformation. 
As the right hand side of (\ref{eqn:SecondLemma}) is constant on each cell of $ \mathcal{LH}(f)$, this set is exactly one connected 
component of $\overline{\mathcal{LA}'}(f)^c$ in $\mathbb{R}^n$. 
Thus, if we consider two points $\theta$ and $\tilde \theta$ in $\mathbb{R}^n$ which both maps to $\Arg(c)B + 2\pi lB$, then we can 
assume that $\theta$ and $\tilde \theta$ fulfils (\ref{eq:SurjNInj1}) for the same numbers $\mu$, however possibly for 
different integers $l$. Under this assumption there are integers $s_1,\dots, s_N$ such that
\[
\langle\alpha_k, \theta \rangle = \langle\alpha_k, \tilde \theta\rangle + 2\pi s_k, \quad k = 1, \dots, N.
\]
The sublattice of $\mathbb{Z}^{n+1}$ generated by the columns of $A$ has $n+1$ generators, and its index is given by the absolute 
value of their determinant. As the determinant is multilinear, this is a linear combination of the determinants of the 
maximal minors of $A$. It follows that the assumption that $g_A=1$ is equivalent to that the columns of $A$ span 
$\mathbb{Z}^{n+1}$ over $\mathbb{Z}$. Thus, for each vector $e_i$ there are integers $r_i= (r_{i1}, \dots, r_{iN})$ such that 
$e_i = \sum_k r_{ik} \alpha_k$. Hence,
\[
\theta_i = \langle e_i, \theta\rangle = \sum_{k=1}^N r_{ik}\langle\alpha_k, \theta\rangle = \sum_{k=1}^N r_{ik}\langle\alpha_k, \tilde\theta\rangle + 2\pi r_{ik} s_k = \tilde\theta_i + 2\pi\langle r_i,s\rangle,
\]
which shows that $\theta$ and $\tilde \theta$ correspond to the same point in $\mathbf{T}^n$.
\end{proof}

\begin{remark}
In general, the map $\cord$ will be $g_A$ to one. Thus, if one considers $\cord$ as a map from $\ccc(\overline{\mathcal{LA}'}(f))$ into 
the full translated lattice $\inter(\mathcal{Z}_B)\cap (\Arg_\pi(c)B + 2\pi \mathbb{Z}^m)$, then injectivity is measured in terms of $g_A$,
while surjectivity is measured in terms of $g_B$. 
In view of Corollary \ref{cor:CoAmoebaTransformation}, if one is interested in the structure of the set of connected 
components of the complement of the closed coamoeba, it is natural to assume that $\cord$ is a bijection.
\end{remark}

\begin{remark}
 The order of a component $\Theta$ of the complement of $\overline{\mathcal{LA}'}(f)$ is most easily determined using the righ hand side of
 \eqref{eqn:EvalutaionOfCoord}. In particular, if the constant $1$ is a monomial of $f$, then
 \[
  \cord(\Theta) = v_0(\theta) = (\arg_\pi(c_{\alpha_1}e^{i\langle\alpha_1, \theta\rangle}), \dots, \arg_\pi(c_{\alpha_N}e^{i\langle\alpha_N, \theta\rangle}))\,B, \quad \theta\in \Theta.
 \]
\end{remark}

\begin{example}
Let us determine the map $\cord$ explicitly in the first example shown in Figure \ref{fig:LopsidedExample}, that is we 
consider the polynomial $f(z_1,z_2) = z_1^3 + z_2 + z_2^2 - z_1 z_2$. The point configuration is
\[
 A = \left(\begin{array}{cccc}
            1 & 1 & 1 & 1 \\
            3 & 0 & 0 & 1 \\
            0 & 1 & 2 & 1
           \end{array}\right),
\]
and a Gale dual of $A$ is given by
\[
 B = (-1,-1,-1,3)^t.
\]
The corresponding zonotope is the interval $\mathcal{Z}_B = [-3\pi, 3\pi]$. As the translation $\Arg_\pi(c)B = 3\arg_\pi(-1) = 3\pi$,
the image of the map $\cord$ will be the doubleton $\{-\pi, \pi\}$. 
To determine $\cord$, it is enough to evaluate $v_\alpha$ for some $\alpha$ and one point in each of the two connected 
components of $\overline{\mathcal{LA}'}(f)^c$, and we see from the picture in Figure \ref{fig:LopsidedExample} that a natural choice 
of points is $\theta_1 = (-2\pi/3,0)$ and $\theta_2 = (2\pi/3, 0)$. We find that
\[
\begin{array}{rcrcr}
v_{\alpha_1}(\theta_1) & = & (0,-2\pi,-2\pi,-\pi)B & = &\pi\phantom{.}\\
v_{\alpha_1}(\theta_2) & = & (0, 2\pi, 2\pi, \pi) B & = &-\pi.
\end{array}
\]
\end{example}

\begin{example}
Let us also consider a univariate case of codimension $1$, namely
\[
f(z) = 1 + z^3 + i z^5.
\]
A Gale dual of $A$ is given by $B = (2, -5, 3)^t$, hence the zonotope is the interval $\mathcal{Z}_B = [-5\pi, 5\pi]$. As the 
translation term is $(0,0,\pi/2)B = 3\pi/2$, the image of $\cord$ is $\{-9\pi/2, -5\pi/2,-\pi/2,3\pi/2, 7\pi/2\}$. 
The lopsided coamoeba $\mathcal{LA}'(f)$ can be seen in Figure \ref{fig:onedimex}. 
\begin{figure}[h]
\centering
\includegraphics[width=60mm]{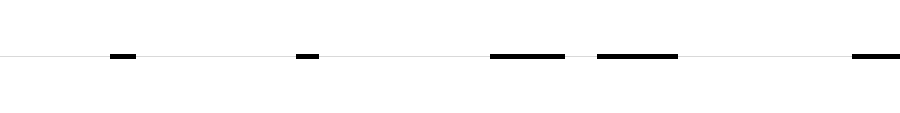}
\caption{$\mathcal{LA}'(f)$ in the fundamental domain $[-\pi, \pi]$.}
\label{fig:onedimex}
\end{figure}
We choose one point from each connected component, namely 
\[
\theta_1= -\frac{7 \pi}8, \quad \theta_2 =  -\frac\pi2,\quad  \theta_3 = 0,\quad \theta_4=  \frac{5 \pi}{16}, \quad \theta_5= \frac{3 \pi}4,
\]
and find that
\[
\begin{array}{rcrcr}
\vspace{2pt} v_0(\theta_1) & = &  (0, -5\pi/8, \pi/8) B & = &7 \pi/2\phantom{.} \\
\vspace{2pt} v_0(\theta_2) & = &  (0, \pi/2, 0)B & = &-5\pi/2\phantom{.} \\
\vspace{2pt} v_0(\theta_3) & = &  (0, 0, \pi/2)B & = &3\pi/2\phantom{.} \\
\vspace{2pt} v_0(\theta_4) & = &  (0, 15\pi/16, \pi/16)B & = &-9\pi/2\phantom{.} \\
\vspace{2pt} v_0(\theta_5) & = &  (0, \pi/4, \pi/4)B & = & -\pi/2.
\end{array}
\]
It is notable that the orders do not reflect the positions of the connected components of the complement on $\mathbf{T}$.
\end{example}

Let us make a short sidestep and consider the non-closed lopsided coamoeba, $\mathcal{LA}'(f)$. The map $\cord$ extends to a map on 
$\ccc(\mathcal{LA}'(f))$ if one allows the image to contain points on the boundary of $\mathcal{Z}_B$. However, the vertices of $\mathcal{Z}_B$ will 
not lie in the image of this map.

\begin{theorem}
\label{thm:OpenMapToZonotope}
Let $f$ be a Laurent polynomial, and let $B$ be a dual matrix of $A$. Then the map $\cord$ can be extended to a surjective 
map
\[
\cord\colon \ccc(\mathcal{LA}'(f)) \rightarrow (\mathcal{Z}_{B}\setminus \ver(\mathcal{Z}_B))\cap (\Arg(c)B+ 2\pi\mathbb{Z}[B]),
\]
where $\ver(\mathcal{Z}_B)$ denotes the set of vertices of $\mathcal{Z}_B$. If $g_A=1$ then this map is an injection.
\end{theorem}

\begin{proof}
The proof is by following the same steps as in the proofs of Theorem \ref{thm:MapToZonotope}, and Propositions 
\ref{thm:Surjectivity} and \ref{thm:Injectivity}, with the only difference that we allow for $|\mu_i|\leq 1$. 
We only note that $p$ is a vertex of $\mathcal{Z}_B$ if and only if any $\mu\in [-1,1]^N$ such that $p = \pi \mu B/2$ has 
$|\mu_k|=1$ for each $k$. This implies that $f\langle\theta\rangle$ is contained in one line (but not in an open half-space), and 
hence that $\theta\in \mathcal{LA}'(f)$.
\end{proof}

Hence we also have a description of the set $\ccc({\mathcal{LA}'(f)})$, where we note especially that the bound $n! \Vol(\Delta_f)$ 
does not hold for $|\ccc(\mathcal{LA}'(f))|$, as shown in the following example.
\begin{example}
Considering the point configuration $A$, with Gale dual $B$, given by
\[
A = \left( \begin{array}{ccccc} 1 & 1 & 1 & 1 & 1 \\ 0 & 1 & 0 & 2 & 3 \\ 0 & 0 & 1 & 1 & 0 \end{array}\right)
\quad \text{and} \quad
B = \left( \begin{array}{cc} 2 & 2 \\ -2 & -3 \\ -1 & 0\\ 1 & 0\\ 0 & 1\end{array}\right).
\]
It is straightforward to check that the coefficients $c = (1,1,1,1, -1)$ yield that the set 
$(\mathcal{Z}_{B}\setminus \ver(\mathcal{Z}_B))\cap (\Arg(c)B+ 2\pi\mathbb{Z}^2)$ contains $6$ elements, while $2!\Vol(\Delta_f) = 5$.
\end{example}
 However, we should remark that the corresponding result to Proposition \ref{pro:ComplementInjection} also fails, leaving 
 the question of whether the normalized volume of the Newton polytope is the correct bound also for $|\ccc(\mathcal{A}'(f))|$ as an 
 open problem.

\section{Coamoebas of polynomials of small codimension}
When $A$ is simple the coamoeba $\mathcal{A}'(f)$ is well known, and as noted earlier $\mathcal{A}'(f) = \mathcal{LA}'(f)$.
Let us now consider coamoebas of polynomials of codimension one and two.

\subsection{Circuits}
Consider the case of codimension one, imposing also the assumption that all maximal minors of $A$ are non-vanishing.
In particular $A$ is a \emph{circuit}, an important special case treated exhaustively in \cite[Chap. 7.1B]{GKZ1}.
As before, we can write $A$ in the form
\[
 A = \left(\begin{array}{ccc}
            1&1&1\\
            0& A_{\mathbf{1}} & \alpha_{n+1}\\
           \end{array}\right),
\]
where $\det(A_{\mathbf{1}}) \neq 0$.

\begin{lemma}
\label{lem:CurcuitLemma}
If $A$ is a circuit, then a dual matrix of $A$ is given by the column vector 
\[
B = (\det(A_{\hat 0}),\,-\det(A_{\hat 1}),\,\dots,\,(-1)^n\det(A_{\hat{n}}),\,(-1)^{n+1}\det(A_{\hat{n+1}}))^t,
\]
where $A_{\hat j}$ denotes the $(n+1)\times(n+1)$-matrix obtained by removing the $j$th column from $A$. 
\end{lemma}

\begin{proof}
Let us use the notation
\[
\det(\hat A) = (\det(A_{\hat 0}),\,-\det(A_{\hat 1}), \,\dots,\,(-1)^n\det(A_{\hat n}), (-1)^{n+1}\det(A_{\hat {n+1}}))^t.
\]
We can write $A = TM$, where
\[
T = \left( \begin{array}{cc} 1 & 0 \\ 0 & A_{\mathbf{1}} \end{array}\right),\quad \text{and}\quad M = \left( \begin{array}{ccc} 1 & 1 & 1 \\ 0 & I_n & \beta \end{array}\right),
\]
with $\beta = A_{\mathbf{1}}^{-1}\alpha_{n+1} \in \mathbb{Q}^{n}$. It is straightforward to check that $M \det(\hat M) = 0$, which 
implies that
\[
A \det(\hat A) = T M \det(\hat M) \det(T) = 0.
\]
As $\det(\hat A)$ is an integer vector, it follows that it is a dual matrix of $A$.
\end{proof}

\begin{theorem}
\label{thm:Circuits}
Let $A$ be a circuit. Then $\overline{\mathcal{LA}'}(f)$, and hence also $\overline{\mathcal{A}'}(f)$, has $n! \Vol(\Delta_f)$ many complement 
components for generic coefficients.
\end{theorem}

\begin{proof}
Let $\Vol(A_{\hat j})$ denote the normalized volume of the simplex $A_{\hat j}$, that is $n!$ times its Euclidean volume.
Then $|\det(A_{\hat j})| = \Vol(A_{\hat j})$.
Using the dual matrix given in Lemma \ref{lem:CurcuitLemma}, we find that the zonotope $\mathcal{Z}_B$ is an interval of length
$\pi(\Vol(A_{\hat 0})+\dots +  \Vol(A_{\hat {n+1}}))$, and it follows from \cite[Chap. 7, Prop. 1.2, p.217]{GKZ1} that
\[
\pi(\Vol(A_{\hat 0})+\dots +  \Vol(A_{\hat {n+1}})) = 2\pi n!\Vol(\Delta_f).
\]
The components of $B$ are the maximal minors of $A$, and hence $g_A = g_B$, both which we can assume equals 1. 
We see that for generic coefficients
\[
\big|\inter(\mathcal{Z}_B)\cap (\Arg(c)B + 2\pi \mathbb{Z})\big| = n!\Vol(\Delta),
\]
and conclude the theorem from Propositions \ref{thm:Surjectivity} and \ref{thm:Injectivity}.
\end{proof}

It was conjectured by Passare \cite[Conjecture 8.1]{CK} that if $A$ is maximally sparse, then the maximal number of 
connected components of the complement of the closed coamoeba is obtained for generic coefficients. In general this 
conjecture is false, with counterexamples given already in the text \cite{CK}. However, we can conclude that the conjecture 
is true in the following special case.

\begin{corollary}
\label{cor:Passare}
If the Newton polytope $\Delta_f$ has $n+2$ vertices, then the upper bound $n!\Vol(\Delta_f)$ on the number of connected 
components of the complement of the coamoeba $\overline{\mathcal{A}'}(f)$ is obtained for maximally sparse polynomials with generic 
coefficients.
\end{corollary}

\begin{proof}
Using the previous theorem, it is enough to show that if $f$ is maximally sparse, then $A$ is a circuit. Indeed, as all 
points in $A$ are vertices of $\Delta_f$, we find that any choice of $n+1$ points will span a simplex of full dimension, 
whence the corresponding determinant is non-vanishing.
\end{proof}

When $n\geq 2$, and for generic coefficients, the topological equivalence between $\overline{\mathcal{A}'}(f)$ and $\overline{\mathcal{LA}'}(f)$ 
implied by Theorem \ref{thm:Circuits} also yields a method to construct a set of \emph{base points} for the set of connected 
components of the complement of the coamoeba, by which we mean a set with exactly one element in each such component. 
Given a polynomial 
\[
 f(z) = c_0 + c_1 z^{\alpha_1} + \dots + c_n z^{\alpha_n} + c_{n+1} z^{\alpha_{n+1}},
\]
under the above assumptions, consider the $n$ polynomials given by
\[
 f_i(z) = f(z) - nc_iz^{\alpha_i} - 2 c_{n+1}z^{\alpha_{n+1}}, \quad i= 1, \dots, n,
\]
and the system
\[
f_1(z) = \dots = f_n(z) = 0.
\]
Note that since $n\geq 2$ we have that $\Delta_{f_i} = \Delta_f$ for each $i$. Avoiding the discriminant locus of this 
system, the BKK theorem \cite[Chap. 6, Thm. 2.2, p.201]{GKZ1} tells us that such a system has exactly $n!\Vol(\Delta_f)$ distinct solutions 
in $(\mathbb{C}^*)^n$. Let $S$ be the set of arguments of these solutions. The above system is equivalent to
\begin{equation}
\label{eq:CodimOneBinomials}
 \left\{\begin{array}{lcll}
         c_1z^{\alpha_1} - c_iz^{\alpha_i} & = & 0 & \quad i=2, \dots,  n\\
         c_0 - c_{n+1} z^{\alpha_{n+1}} & = & 0, &
        \end{array}\right.
\end{equation}
which shows that for each $\theta\in S$  
the set $f\langle\theta\rangle$ contains at most two points. Thus, under the genericity assumption $f\langle\theta\rangle$ is lopsided for 
each $\theta\in S$. 
It also follows that $|S| = n!\Vol(\Delta_f)$, and that the numbers 
\[
\phi_\theta = \arg_\pi\bigg(\frac{c_1 e^{i\langle\alpha_1,\theta\rangle}}{c_0}\bigg)= \dots = \arg_\pi\bigg(\frac{c_n e^{i\langle\alpha_n,\theta\rangle}}{c_0}\bigg), \quad \theta \in S
\]
are distinct. Hence, the orders
\[
\cord(\Theta) = \phi_\theta\,(0,1,\dots, 1,0) B
\]
are also distinct. We conclude that $S$ has exactly one element in each connected component of $\overline{\mathcal{A}'}(f)^c$.

\subsection{The case $m=2$ and a relation to discriminants}
\label{ssec:Case2}

Let us move up one step in the complexity chain and consider the case when $m=2$. We will assume that $g_A=1$. Related to 
the point configuration $A$ is the so-called $A$-discriminant $D_A(c)$, which is a polynomial in the coefficients $c$ 
vanishing if and only if the hypersurface $V(f)\subset (\mathbb{C}^*)^n$ is singular, see \cite{GKZ1}. 
The polynomial $D_A(c)$ enjoys one homogeneity relation for each row of the matrix $A$, and choosing a Gale dual of $A$ 
yields a dehomogenization of $D_A(c)$ in the following manner; introducing the variables 
\[
x_i = c_1^{b_{1i}}\cdots c_N^{b_{Ni}}, \quad i=1, \dots, m,
\] 
then there is a Laurent monomial $g(c)$, such that $g(c)D_A(c) = D_B(x)$.
This relation is described in more detail in \cite{NP}, where it was first shown that the zonotope $\mathcal{Z}_B$ together with 
the coamoeba 
$\mathcal{A}'(D_B)$ of the dehomogenized discriminant generically covers $\mathbf{T}^2$
precisely $n!\Vol(\Delta_f)$ many times. Hence, if $\overline{\mathcal{A}'}(D_B)\neq \mathbf{T}^2$, then there is a choice of coefficients 
$c$ such that the set $(\Arg(c) + 2\pi \mathbb{Z}^2 )\cap \inter(\mathcal{Z})$ has $n!\Vol(\Delta_f)$ many elements. If so, then we can 
find a coamoeba whose complement has the maximal 
number of connected components. As the next example shows this is not always the case.

\begin{example}
\label{ex:CodimensionTwo}
Consider the point configuration
\[
 A = \left(\begin{array}{ccccc}
            1 & 1 & 1 & 1 & 1 \\
            0 & 2 & 0 & 1 & 2 \\
            0 & 0 & 3 & 3 & 2
           \end{array}\right).
\]
where we note that $2!\Vol(\Delta_f) = 11$. The dehomogenized discriminant related to the Gale dual 
\[
 B = \left(\begin{array}{cc}
            1 &  2 \\
           -1 & -3 \\
           -2 & -2 \\
	    2 &  0 \\
	    0 &  3 \\
           \end{array}\right)
\]
is
\begin{align*}
D_B(x) = & \,\,729 x_1^2 + 2187 x_1^3 + 2187 x_1^4 + 729 x_1^5 + 1728 x_2 + 4752 x_1 x_2\\
                 & +  5400 x_1^2 x_2  - 1404 x_1^3 x_2 - 864 x_1^4 x_2 + 3456 x_2^2 -  5616 x_1 x_2^2\\
                 & + 576 x_1^2 x_2^2 + 256 x_1^3 x_2^2 + 1728 x_2^3.
\end{align*}
Its coamoeba covers the torus $\mathbf{T}^2$ completely, and hence the complement of the closed lopsided coamoeba can not have more 
than 10 connected components.
\begin{figure}[h]
\centering
\includegraphics[width=40mm]{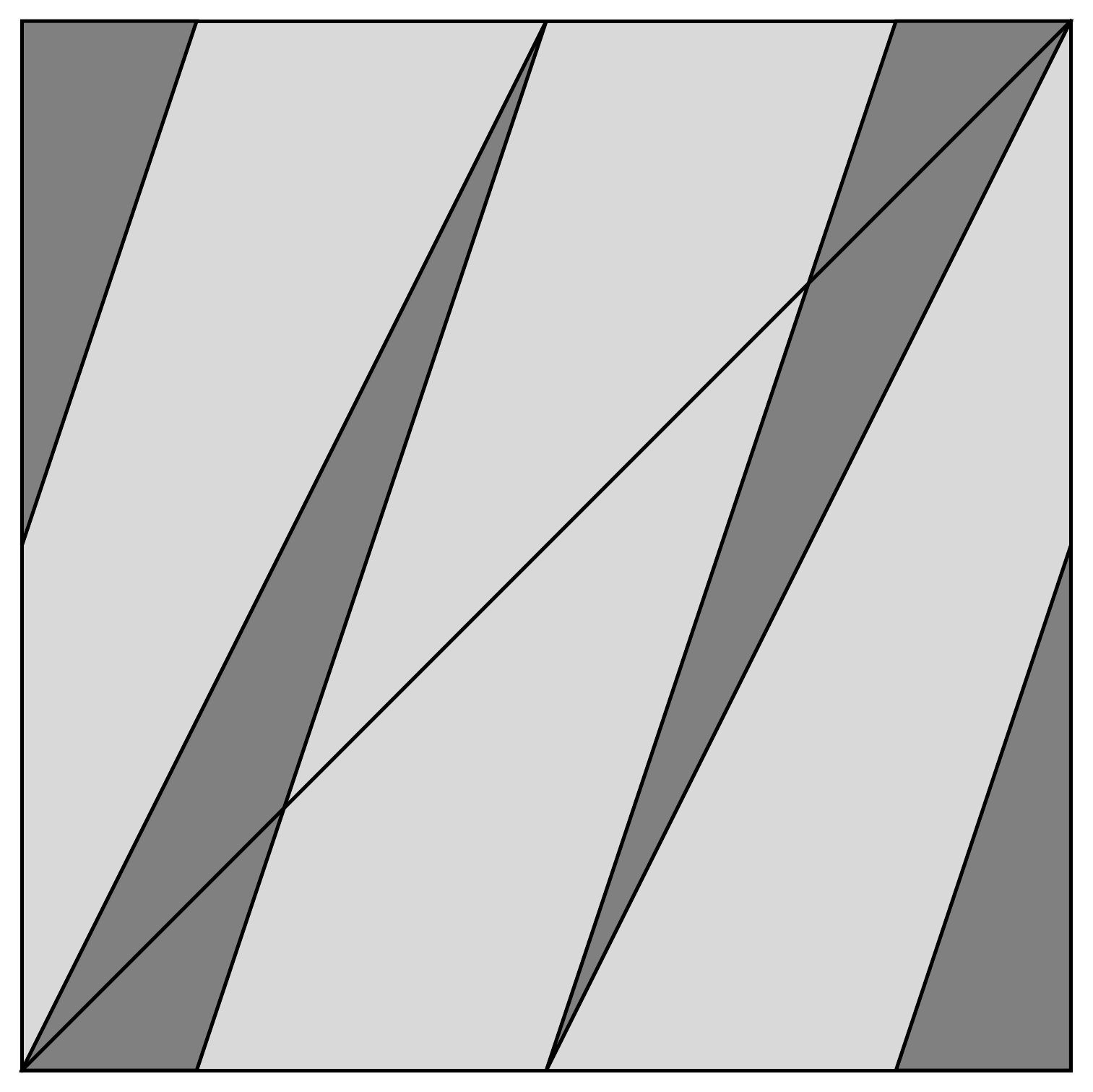}
\caption{The coamoeba of $D_B(x)$ drawn with multiplicity, darker areas are covered twice.}
\label{fig:DiscriminantCovers}
\end{figure}
\end{example}

The connection between the zonotope $\mathcal{Z}_B$ and the dehomogenized discriminant $D_B(x)$ is believed to be true also in 
higher codimensions, however this is still an open problem. For the latest development, we refer the reader to \cite{PS}.

The fact that we cannot always construct a coamoeba whose complement has $n!\Vol(\Delta_f)$ many connected components is of 
course a source of just criticism. However, let us note that it has not been proved that this upper bound is sharp. To the 
contrary, recent examples suggest that this is not the case \cite{For}.

\end{document}